\documentclass[preprint,review,12pt]{elsarticle}

\usepackage{amssymb}
\usepackage{amsthm}

\usepackage{amsmath}
\usepackage{hyperref}

\usepackage{lineno}

\usepackage{epstopdf}
\usepackage{subfigure}

\newtheorem{theorem}{Theorem}
\newtheorem{lemma}[theorem]{Lemma}
\newtheorem{assumption}{Assumption}
\newtheorem{remark_temp}[theorem]{Remark}
\newenvironment{remark}{\begin{remark_temp} \upshape}{ \end{remark_temp}}

\journal{Acta Astronautica}

\begin{document}

\begin{frontmatter}



\title{Robust Three-axis Attitude Stabilization for Inertial Pointing Spacecraft Using Magnetorquers}


\author{Fabio Celani}

\address{Department of Astronautical, Electrical, and Energy Engineering\\
Sapienza University of Rome\\
Via Salaria 851, 00138 Roma, Italy\\
fabio.celani@uniroma1.it}

\begin{abstract}

In this work feedback control laws are designed for achieving three-axis attitude stabilization of inertial pointing spacecraft using only magnetic torquers. 
The designs are based on an almost periodic model of geomagnetic field along the spacecraft's orbit.  Both attitude plus  attitude rate feedback, and attitude only feedback  are proposed. Both feedback laws achieve local exponential stability robustly with respect to large uncertainties in the spacecraft's inertia matrix. The latter properties are proved using general averaging and Lyapunov stability.  Simulations are included to validate the effectiveness of the proposed control algorithms.

\end{abstract}


\end{frontmatter}



\section{Introduction}

SpacecraftÕs attitude control can be obtained by adopting several mechanisms. Among them electromagnetic actuators are widely used for generation of attitude control torques on small satellites flying low Earth orbits. They consist of planar current-driven coils rigidly placed on the spacecraft typically along three orthogonal axes, and they operate on the basis of the interaction between the magnetic moment generated by those coils and the Earth's magnetic field; in fact, the interaction with the Earth's field generates a torque that attempts to align the total magnetic moment in the direction of the field.
The interest in such devices, also known as magnetorquers, is due to the following reasons: (i) they are simple, reliable, and low cost (ii) they need only renewable electrical power to be operated; (iii) using magnetorquers it is possible to modulate smoothly the control torque so that unwanted couplings with flexible modes, which could harm pointing precision, are not induced; (iv) magnetorquers save system weight  with respect to any other class of actuators.  On the other hand, magnetorquers have the important limitation that control torque is constrained to belong to the plane orthogonal to the Earth's magnetic field. As a result, different types of actuators often accompany magnetorquers to provide full three-axis control, and a considerable amount of work has been dedicated to the design of magnetic control laws in the latter setting (see e.g.  \cite{Silani:2005kx, Pittelkau:1993fk,  Lovera:2002uq, pulecchi_cst10} and references therein). 

Recently, three-axis attitude control using only magnetorquers has been considered as a feasible option especially for low-cost micro-satellites. Different control laws have been obtained; many of them are designed using a periodic approximation of the time-variation of the geomagnetic field along the orbit, and in such scenario  stability and disturbance attenuation have been achieved using results from linear periodic systems (see e.g. \cite{Wisniewski:1999uq, Psiaki:2001yu, Reyhanoglu:2009uq}); however, in \cite{Lovera:2004vn} and \cite{Lovera:2005fk}  stability has been achieved even when a non periodic, and thus more accurate, approximation of the geomagnetic field is adopted. In both works 
feedback control laws that require measures of both attitude and attitude-rate (i.e. state feedback control laws) are proposed; moreover, in \cite{Lovera:2004vn}  
feedback control algorithms  which need measures of attitude only (i.e. output feedback control algorithms) are presented, too.
All the control algorithms in \cite{Lovera:2004vn} and \cite{Lovera:2005fk} require exact knowledge of the spacecraft's inertia matrix; however, because the moments and products of inertia of the spacecraft may be uncertain or may change due
to fuel usage and articulation, the inertia matrix of a spacecraft is often subject to large uncertainties; as a result, it is important to determine control algorithms which achieve attitude stabilization in spite of those uncertainties.

In this work we present  control laws  obtained by modifying those in \cite{Lovera:2004vn} and \cite{Lovera:2005fk}, which achieve local exponential stability in spite of large uncertainties on the inertia matrix. The latter results are derived adopting an almost periodic model of the geomagnetic field along the spacecraft's orbit. As in \cite{Lovera:2004vn} and \cite{Lovera:2005fk} the main tools used in the stability proofs are general averaging and Lyapunov stability (see \cite{khalil00}).

The rest of the paper is organized as follows. Section \ref{modeling} introduces the models adopted for the spacecraft and for the Earth's magnetic field. Control design of both state and output feedbacks are reported in Section \ref{control_design} along with stability proofs. Simulations of the obtained control laws are presented in Section 
\ref{simulations}.

\subsection{Notations} For $x \in \mathbb{R}^{n}$, $\|x\|$ denotes the Eucledian norm of $x$; for a square matrix $A$, $\lambda_{min}(A)$ and $\lambda_{max}(A)$ denote the minimum and maximum eigenvalue of $A$ respectively; $\|A\|$ denotes the 2-norm of $A$ which is equal to $\|A\|=[\lambda_{max}(A^T A)]^{1/2}$. Symbol $I$ represents the identity matrix. For $a \in \mathbb{R}^3$, $a^{\times}$ represents the skew symmetric matrix 
\begin{equation} \label{skew_symmetric}
a^{\times} =
\left[
\begin{array}{rcl}
  0 & -a_3 & a_2\\
  a_3 & 0 & -a_1\\
  -a_2 & a_1 & 0
\end{array}
\right]
\end{equation} so that
for $b  \in \mathbb{R}^3$, the multiplication $a^{\times}b$ is equal to the cross product $a \times b$.

\section{Modeling} \label{modeling}

In order to describe the attitude dynamics of an Earth-orbiting rigid spacecraft, and in order to represent the geomagnetic field, it is useful to introduce the following reference frames.

\begin{enumerate}
  \item \emph{Earth-centered inertial frame} $\mathcal{F}_i$. A commonly used inertial frame for Earth orbits is the Geocentric Equatorial Frame, whose origin is in the Earth's center, its $x_i$ axis is the vernal equinox direction, its $z_i$ axis coincides with the Earth's axis of rotation and points northward,  and its $y_i$ axis completes an orthogonal right-handed frame (see \cite[Section 2.6.1]{sidi97} ).
   \item \emph{Spacecraft body frame} $\mathcal{F}_b$. The origin of this right-handed orthogonal frame attached to the spacecraft, coincides with the satellite's center of mass; its axes are chosen so that the inertial pointing objective is having  $\mathcal{F}_b$ aligned with $\mathcal{F}_i$.
\end{enumerate}

Since the inertial pointing objective consists in aligning $\mathcal{F}_b$ to $\mathcal{F}_i$, the focus will be on the relative kinematics and dynamics of the satellite with respect to the inertial frame. Let 
$q=[q_1\ q_2\ q_3\ q_4]^T=[q_v^T\ q_4]^T$
with $\|q\|=1$ be the unit quaternion representing rotation of  $\mathcal{F}_b$ with respect to $\mathcal{F}_i$; then, the corresponding attitude matrix  is given by 
 \begin{equation} \label{A_bo}
A(q)=(q_4^2-q_v^T q_v)I+2 q_v q_v^T-2 q_4 q_v^{\times}
\end{equation}
(see \cite[Section 5.4]{wie98}).

Let
\begin{equation} \label{w}
W(q)=\dfrac{1}{2}\left[\begin{array}{c}q_4 I + q_v^{\times} \\-q_v^T\end{array}\right]
\end{equation}
Then the relative attitude kinematics is given by 
\begin{equation} \label{kinem_body}
\dot q = W(q) \omega
\end{equation} where $\omega \in \mathbb{R}^{3}$ is the angular rate of $\mathcal{F}_b$ with respect to $\mathcal{F}_i$ resolved in $\mathcal{F}_b$ (see \cite[Section 5.5.3]{wie98}).

The attitude dynamics in body frame can be expressed by 
\begin{equation} \label{euler_eq}
J \dot \omega =- \omega^{\times} J \omega + T
\end{equation} where $J \in \mathbb{R}^{3 \times 3}$ is the spacecraft inertia matrix, and $T \in \mathbb{R}^{3}$ is the vector of external torque expressed in $\mathcal{F}_b$ (see \cite[Section 6.4]{wie98}). As stated in the introduction, here we consider $J$ uncertain since the moments and products of inertia of the spacecraft may be uncertain or may change due to fuel usage and articulation; however, we require to know a lower bound and an upper bound for the spacecraft's principal moments of inertia; those bounds usually can be determined in practice without difficulties. Thus, the following assumption on $J$ is made.

\begin{assumption} \label{uncetain_inertia}
The inertia matrix $J$ is unknown, but  bounds $0<J_{min} \leq J_{max}$ such that the following hold
\begin{equation} \label{J_min}
  0 <J_{min} \leq \lambda_{min}(J) \leq \lambda_{max}(J)=\|J\| \leq J_{max}
\end{equation} are known.
\end{assumption}

The spacecraft is equipped with three magnetic coils aligned with the $\mathcal{F}_b$ axes which generate the magnetic attitude control torque
\begin{equation} \label{t_coils}
T=m_{coils} \times B^b=-B^{b \times}\ m_{coils}
\end{equation} where $m_{coils} \in \mathbb{R}^3$ is the vector of magnetic moments for the three coils, and $B^b$ is the geomagnetic field at spacecraft expressed in body frame $\mathcal{F}_b$. From the previous equation, we see that magnetic torque can only be perpendicular to geomagnetic field.

Let $B^i$ be the geomagnetic field at spacecraft expressed in inertial frame $\mathcal{F}_i$. Note that $B^i$ varies with time both because of the spacecraft's motion along the orbit and because of time variability of the geomagnetic field. Then
$B^b(q,t)=A(q)B^i(t)$
which shows explicitly the dependence of $B^b$ on both $q$ and $t$.

Grouping together equations (\ref{kinem_body}) (\ref{euler_eq}) (\ref{t_coils}) the following nonlinear time-varying system is obtained
\begin{equation} \label{nonlin_tv}
\begin{array}{rcl}
  \dot q &=& W(q) \omega\\
  J \dot \omega &=& - \omega^{\times} J \omega -B^{b}(q,t)^{\times}\ m_{coils} \end{array}
\end{equation} in which $m_{coils}$ is the control input.

In order to design control algorithms, it is important to characterize the time-dependence of $B^b(q,t)$ which is the same as  characterizing the time-dependence of $B^i(t)$. Adopting the so called dipole model of the geomagnetic field (see \cite[Appendix H]{wertz78}) we obtain  \begin{equation} \label{geomegnetic_model}
B^i(t)=\frac{\mu_m}{\|r^i(t)\|^3}[3( \hat m^i(t)^T \hat{r}^i(t)) \hat{r}^i - \hat m^i(t)  ]
\end{equation}
In equation (\ref{geomegnetic_model}), $\mu_m$
is the total dipole strength, $r^i(t)$ is the spacecraft's position vector resolved in $\mathcal{F}_i$, and $\hat r^i(t)$ is the vector of the direction cosines of $r^i(t)$; finally $\hat m^i(t)$ is the vector of the direction cosines of the Earth's magnetic dipole expressed in $\mathcal{F}_i$  which is set equal to
\begin{equation} \label{m_hat}
\hat m^i(t) = \left[\begin{array}{c} \sin(\theta_m) \cos(\omega_{e}t + \alpha_0) \\ \sin(\theta_m) \sin(\omega_{e}t + \alpha_0) \\ \cos(\theta_m) \end{array}\right]
\end{equation} where $\theta_m$ is the dipole's coelevation, 
$\omega_{e}=360.99$ deg/day is the Earth's average rotation rate, and $\alpha_0$ is the right ascension of the dipole at time $t=0$; clearly, in equation (\ref{m_hat}) Earth's rotation has been taken into account. It has been obtained that for year 2010 $\mu_m=7.746\ 10^{15}$ Wb m and $\theta_m=170.0^{\circ}$ (see \cite{Rodriguez-Vazquez:2012fk}); then, as it is well known, the Earth's magnetic dipole is tilted with respect to Earth's axis of rotation.

Equation (\ref{geomegnetic_model}) shows that in order to characterize the time dependence of $B^i(t)$ it is necessary to determine an expression for  $r^i(t)$ which is the spacecraft's position vector resolved in $\mathcal{F}_i$.  Assume that the orbit is circular, and define a coordinate system $x_p$, $y_p$ in the orbital's plane whose origin is at Earth's center; then, the position of satellite's center of mass is clearly given by
\begin{equation} \label{x_p_y_p}
\begin{array}{rcl}
  x^p(t) &=& R \cos(nt + \phi_0)\\
  y^p(t) &=& R \sin(nt + \phi_0)
\end{array}
\end{equation} where $R$ is the radius of the circular orbit, $n$ is the orbital rate, and $\phi_0$ an initial phase. Then, coordinates of the satellite in inertial frame $\mathcal{F}_i$ can be easily obtained from (\ref{x_p_y_p}) using an appropriate rotation matrix which depends on the orbit's inclination $incl$ and on the right ascension of the ascending node $\Omega$ (see \cite[Section 2.6.2]{sidi97}). Plugging  into (\ref{geomegnetic_model}) the expression of those coordinates and equation (\ref{m_hat}), an explicit expression for $B^i(t)$ can be obtained; it can be easily checked that $B^i(t)$ turns out to be a linear combination of sinusoidal functions of $t$ having different frequencies. As a result, $B^i(t)$ is an almost periodic function of $t$ (see \cite[Section 10.6]{khalil00}), and consequently system (\ref{nonlin_tv}) is an almost periodic nonlinear system.

\section{Control design} \label{control_design}

As stated before, the control objective is driving the spacecraft so that $\mathcal{F}_b$ is aligned with $\mathcal{F}_i$. From (\ref{A_bo}) it follows that $A(q)=I$ for $q=[q_v^T\ q_4]^T=\pm \bar q$ where $\bar q= [0\ 0\ 0\ 1]^T$. Thus,  the objective is designing control strategies for $m_{coils}$ so that $q_v \rightarrow 0$ and $\omega \rightarrow 0$. Here we will present feedback laws that locally exponentially stabilize equilibrium $(q,\omega)=(\bar q, 0)$.

First, since $B^b$ can be measured using magnetometers, apply the following preliminary control which enforces that  $m_{coils}$ is orthogonal to $B^b$
\begin{equation} \label{m_coils}
m_{coils}=B^b(q,t) \times u= B^{b}(q,t)^{\times} u=- (B^{b}(q,t)^{\times})^ T u
\end{equation} where $u \in \mathbb{R}^3$ is a new control vector.  Then, it holds that
\begin{equation} \label{omega_dot}
\begin{array}{rcl}
  \dot q &=& W(q) \omega\\
  J \dot \omega &=& - \omega^{\times} J \omega + \Gamma^b(q,t) u
\end{array}
\end{equation} where
\begin{equation} \label{gamma_b}
\Gamma^b(q,t)= (B^{b}(q,t)^{\times}) (B^{b }(q,t)^{\times})^ T=B^b(q,t)^T B^b(q,t) I-B^b(q,t) B^b(q,t)^T
\end{equation}
Let
\begin{equation} \label{gamma_i}
\Gamma^i(t)=(B^{i}(t)^{\times}) (B^{i}(t)^{\times})^ T=B^i(t)^T B^i(t) I-B^i(t) B^i(t)^T
\end{equation}
then it is easy to verify that
\begin{equation*}
\Gamma^b(q,t)=A(q) \Gamma^i(t) A(q)^T
\end{equation*} 
so that (\ref{omega_dot}) can be written as
\begin{equation} \label{dynamic_o}
\begin{array}{rcl}
  \dot q &=& W(q) \omega\\
  J \dot \omega &=& - \omega^{\times} J \omega + A(q) \Gamma^i(t) A(q)^T u
\end{array}
\end{equation}

Since $B^i(t)$ is a linear combination of sinusoidal functions of $t$ having different frequencies, so is $\Gamma^i(t)$. As a result, the following average
\begin{equation} \label{gamma_i_av}
\Gamma^i_{av}= \lim_{T \rightarrow \infty} \frac{1}{T} \int_{0}^{T}\Gamma^i(\tau) d \tau
\end{equation} is well defined. Consider the following assumption on $\Gamma^i_{av}$ .

\begin{assumption} \label{assumption_gamma_i} The spacecraft's orbit satisfies condition
$\Gamma^i_{av}>0$.

 \end{assumption}

\begin{remark} \label{discuss_assumption2} Since $\Gamma^i(t) \geq 0$ (see (\ref{gamma_i})), Assumption \ref{assumption_gamma_i} is equivalent to requiring that $\det(\Gamma^i_{av}) \neq 0$.
The expression of $\det(\Gamma^i_{av})$ based on the model of the geomagnetic field presented in the previous section is quite complex, and it is not easy to get an insight from it; however, if coelevation of Earth's magnetic dipole $\theta_m=170.0^{\circ}$ is approximated to 
$\theta_m=180^{\circ}$ deg, which corresponds to having Earth's magnetic dipole aligned with Earth's rotation axis, then the geomagnetic field in a fixed point of the orbit becomes constant with respect to time (see (\ref{geomegnetic_model}) and (\ref{m_hat})); consequently $B^i(t)$, which represents the geomagnetic field along the orbit, becomes periodic, and the expression of $\det(\Gamma^i_{av})$ simplifies as follows
\begin{equation} \label{det_Gamma_o_av_simple}
\det(\Gamma^i_{av})=\frac{9 \mu_m^6}{1024\ R^{18}} [345 - 92 \cos(2 \ incl)+3 \cos(4 \ incl)] \sin(incl)^2
\end{equation}
Thus, in such simplified scenario issues on fulfillment of Assumption \ref{assumption_gamma_i} arise only for low inclination orbits.
\end{remark}

\subsection{State feedback}

In this subsection, a stabilizing static state (i.e. attitude and attitude rate) feedback for system (\ref{dynamic_o}) is presented. It is obtained as a simple modification of the one proposed in  \cite{Lovera:2005fk}. The important property that is achieved through such modification is robustness with respect to uncertainties on the inertia matrix; that is, the  modified control algorithm achieves stabilization for all $J$'s that fulfill Assumption \ref{uncetain_inertia}.

\begin{theorem} \label{exponential_stability}
Consider the magnetically actuated spacecraft 
described by (\ref{dynamic_o}) with uncertain inertia matrix $J$ satisfying Assumption \ref{uncetain_inertia}. Apply the following proportional derivative control law
\begin{equation} \label{control_u}
u=-(\epsilon^2 k_1 q_v + \epsilon k_2 \omega)
\end{equation}
with $k_1>0$ and $k_2>0$. Then, under Assumption \ref{assumption_gamma_i}, there exists $\epsilon^*>0$ such that for any $0<\epsilon<\epsilon^*$,  equilibrium $(q,\omega)=(\bar q, 0)$ is locally exponentially stable for (\ref{dynamic_o}) (\ref{control_u}).
\end{theorem}

\begin{proof}

In order to prove local exponential stability of equilibrium $(q,\omega)=(\bar q, 0)$, it suffices considering the restriction of  (\ref{dynamic_o}) (\ref{control_u}) to the open set $\mathbb{S}^{3+} \times \mathbb{R}^3$ where 
\begin{equation} \label{S3}
\mathbb{S}^{3+}=\{ q \in \mathbb{R}^4\  |\  \|q\|=1,\ q_{4}>0\} 
\end{equation} On the latter set the following holds
\begin{equation} \label{z14}  
q_{4}=(1-q_{v}^T q_{v})^{1/2}
\end{equation} Consequently, the restriction of (\ref{dynamic_o}) (\ref{control_u}) to  $\mathbb{S}^{3+} \times \mathbb{R}^3$ is given by the following reduced order system
\begin{equation} \label{closed_loop_reduced}
\begin{array}{rcl}
  \dot q_{v} &=& W_v(q_{v}) \omega\\
  J \dot \omega &=& - \omega^{\times} J \omega - A_v(q_v) \Gamma^i(t) A_v(q_v)^T (\epsilon^2 k_1 q_v + \epsilon k_2 \omega)
\end{array}
\end{equation}
where
\begin{equation} \label{w_v}
W_v(q_{v})=\dfrac{1}{2} \left[ \left( 1-q_{v}^T q_{v} \right)^{1/2} I + q_{v}^{\times} \right]
\end{equation} and
 \begin{equation} \label{a_v}
A_v(q_{v})=\left( 1-2 q_{v}^T q_{v} \right) I+2 q_{v} q_{v}^T-2 \left( 1- q_{v}^T q_{v} \right)^{1/2} q_{v}^{\times}
\end{equation}

Consider the linear approximation of (\ref{closed_loop_reduced})  around $(q_v,\omega)=(0, 0)$ which is given by
\begin{equation} \label{linearized_closed_loop_reduced}
\begin{array}{rcl}
  \dot q_{v} &=& \dfrac{1}{2} \omega\\
   \dot \omega &=& - J^{-1} \Gamma^i(t) (\epsilon^2 k_1 q_v + \epsilon k_2 \omega)
\end{array}
\end{equation}

Introduce the following  state-variables' transformation 
\begin{equation*}
  z_1=q_v \ \ \ \  z_2=\omega /\epsilon
\end{equation*} with $\epsilon>0$ so that system (\ref{closed_loop_reduced}) is transformed into
\begin{equation} \label{linearized_time_varying_s}
\begin{array}{rcl}
  \dot z_1 &=& \dfrac{\epsilon}{2} z_2\\
  \dot z_2 &=& - \epsilon J^{-1}  \Gamma^i(t) ( k_1 z_1 +  k_2 z_2)
 \end{array}
\end{equation}

Rewrite system (\ref{linearized_time_varying_s}) in the following matrix form
\begin{equation} \label{time_var_syst}
\dot z = \epsilon A(t) z
\end{equation} where
\begin{equation*}
A(t)=\left[\begin{array}{cc}0 & \frac{1}{2} I \\ -k_1 J^{-1}  \Gamma^i(t)  & -k_2 J^{-1}  \Gamma^i(t)
\end{array}\right]
\end{equation*} and consider the so called time-invariant ``average system''  of (\ref{time_var_syst})
 \begin{equation}  \label{aver_syst}
\dot z = \epsilon A_{av} z 
\end{equation} with
\begin{equation*}
A_{av}=\left[\begin{array}{cc}0 & \frac{1}{2} I \\ -k_1 J^{-1}  \Gamma^i_{av}  & -k_2 J^{-1}  \Gamma^i_{av}
\end{array}\right]
\end{equation*} where $\Gamma^i_{av}$ was defined in (\ref{gamma_i_av}) (see \cite[Section 10.6]{khalil00} for a general definition of average system).

We will show that after having performed an appropriate coordinate transformation, system (\ref{time_var_syst}) can be seen as a perturbation of (\ref{aver_syst}) (see \cite[Section 10.4]{khalil00}). For that purpose note that since $\Gamma^i(t)$ is a linear combination of sinusoidal functions of $t$ having different frequencies, then there exists $k_{\Delta}>0$ such that the following holds
\begin{equation*}
\left \|  \dfrac{1}{T} \int_0^T  \Gamma^i(\tau) d\tau - \Gamma^i_{av} \right \| \leq k_{\Delta} \dfrac{1}{T} \ \ \ \ \forall \ T>0
\end{equation*}  Let
\begin{equation*}
\Delta(t)= \int_0^t (\Gamma^i(\tau)-\Gamma^i_{av}) d\tau
\end{equation*} then for $t>0$
\begin{equation*}
\left \| \Delta(t) \right\|=t \left \| \left[ \frac{1}{t} \int_0^t \Gamma^i(\tau) d\tau -\Gamma^i_{av} \right] \right\| \leq k_{\Delta} 
\end{equation*} hence
\begin{equation} \label{Delta_bound}
\left \| \Delta(t) \right\| \leq k_{\Delta} \ \ \ \ \forall t \geq 0
\end{equation}
Let
\begin{equation} \label{matrix_U}
U(t)=\int_0^t [A(\tau)-A_{av}] d \tau = \left[\begin{array}{cc}0 & 0 \\ -k_1 J^{-1}  \Delta(t)  & -k_2 J^{-1}  \Delta(t)
\end{array}\right]
\end{equation} and observe that the following holds
\begin{equation} \label{bound_on_U1}
\| U(t) \| \leq \sqrt{3}\ (k_1+k_2) \| J^{-1} \| \  \| \Delta(t) \| \ \ \ \ \forall t \geq 0
\end{equation} Observe that from (\ref{J_min}) it follows that
\begin{equation} \label{J_inv}
\| J^{-1} \| = \frac{1}{\lambda_{min}(J)} \leq \frac{1}{J_{min}}
\end{equation} thus
\begin{equation} \label{bound_on_U}
\| U(t) \| \leq \frac{\sqrt{3}\  (k_1+k_2) k_{\Delta}}{J_{min}} \ \ \ \ \forall t \geq 0
\end{equation}
Now consider the transformation matrix
\begin{equation} \label{def_T}
T(t,\epsilon)=I+\epsilon U(t)=\left[\begin{array}{cc}I & 0 \\ - \epsilon k_1 J^{-1}  \Delta(t)  & I- \epsilon k_2 J^{-1}  \Delta(t)
\end{array}\right]
\end{equation}
Since (\ref{bound_on_U}) holds, if $\epsilon$ is small enough, then $T(t,\epsilon)$ is non singular for all $t \geq 0$.
Thus, we can define the coordinate transformation
\begin{equation*}
w= T(t,\epsilon)^{-1} z
\end{equation*} In order to compute the state equation of system (\ref{time_var_syst}) in the new coordinates it is convenient to consider the inverse transformation
\begin{equation*}
z= T(t,\epsilon)w
\end{equation*} and differentiate with respect to time both sides obtaining
\begin{equation*}
\epsilon A(t)T(t,\epsilon)w=\frac{\partial T}{\partial t}(t,\epsilon) w + T(t,\epsilon) \dot w
\end{equation*} Consequently
\begin{equation} \label{w_dot_1}
\dot w = T(t,\epsilon)^{-1} \left[ \epsilon A(t)T(t,\epsilon)-\frac{\partial T}{\partial t}(t,\epsilon) \right] w
\end{equation}
Observe that
\begin{equation*}
T(t,\epsilon)^{-1}=\left[\begin{array}{cc}I & 0\\
(I-\epsilon k_2 J^{-1} \Delta(t))^{-1}  \epsilon k_1 J^{-1}  \Delta(t) & (I- \epsilon k_2 J^{-1}  \Delta(t) )^{-1}
\end{array}\right]
\end{equation*} By using Lemma \ref{lemma_appendix}, it is immediate to obtain that for $\epsilon$ sufficiently small, matrix $(I- \epsilon k_2 J^{-1}  \Delta(t) )^{-1}$ can be expressed as follows
\begin{equation*}
(I- \epsilon k_2 J^{-1}  \Delta(t))^{-1} = I + \epsilon k_2 J^{-1}  \Delta(t) (I- \tilde \epsilon k_2 J^{-1}  \Delta(t) )^{-2}  
\end{equation*} where $0 < \tilde \epsilon < \epsilon$. As a result $T(t,\epsilon)^{-1}$ can be written as
\begin{equation} \label{T_inv}
T(t,\epsilon)^{-1}=I+\epsilon S(t,\epsilon)
\end{equation} with
\begin{equation*}
S(t,\epsilon)=\left[\begin{array}{cc}0 & 0 \\ (I- \epsilon k_2 J^{-1} \Delta(t))^{-1}   k_1 J^{-1}  \Delta(t)  &  k_2 J^{-1}  \Delta(t) (I- \tilde \epsilon k_2 J^{-1}  \Delta(t) )^{-2} 
\end{array}\right]
\end{equation*} Observe that since (\ref{Delta_bound}) (\ref{J_inv}) (\ref{ineq1}) (\ref{ineq2}) hold,  for $\epsilon$ sufficiently small $S(t,\epsilon)$ is bounded for all $t \geq 0$. Moreover, from (\ref{matrix_U}) and (\ref{def_T}) obtain the following
\begin{equation} \label{dT_deps}
\frac{\partial T}{\partial t}(t,\epsilon)=\epsilon \frac{\partial U}{\partial t}(t,\epsilon)=\epsilon(A(t)-A_{av})
\end{equation} Then, from (\ref{def_T}) (\ref{w_dot_1}) (\ref{T_inv}) (\ref{dT_deps}) we obtain
\begin{equation} \label{perturb_syst}
\dot w = \epsilon [A_{av}+ \epsilon H(t,\epsilon)]w
\end{equation} where
\begin{equation*}
H(t,\epsilon)=A(t)U(t)+S(t,\epsilon) A_{av}+\epsilon S(t,\epsilon) A(t) U(t)
\end{equation*} Thus we have shown that in coordinates $w$ system (\ref{time_var_syst}) is a perturbation of  system (\ref{aver_syst}); moreover,  clearly, for the perturbation factor $H(t,\epsilon)$ it occurs that for $\epsilon$ small enough there exists $k_H>0$ such that
\begin{equation} \label{bound_on_H}
\| H(t,\epsilon) \| \leq k_H \ \ \ \ \forall t \geq 0
\end{equation} 
Let us focus on system
\begin{equation} \label{wdot}
\dot w = A_{av} w
\end{equation} which in expanded form reads as follows
\begin{equation} \label{linearized_time_varying}
\begin{array}{rcl}
  \dot w_1 &=& \dfrac{1}{2} w_2\\
   J \dot w_2 &=&  - \Gamma^i_{av} ( k_1 w_1 +  k_2 w_2)
 \end{array}
\end{equation} Consider the candidate Lyapunov function for system (\ref{linearized_time_varying}) (see \cite{bullo96})
\begin{equation} \label{def_V}
V(w_1,w_2)=k_1 w_1^T \Gamma^i_{av} w_1 + 2 \beta w_1^T J w_2 + \frac{1}{2} w_2^T J w_2
\end{equation} with $\beta>0$. 
Note that
\begin{multline*}
V(w_1,w_2) \geq k_1 \lambda_{min}(\Gamma^i_{av}) \| w_1 \|^2 - 2 \beta \| J \| \| w_1 \| \| w_2 \| + \frac{1}{2} \lambda_{min}(J) \| w_2 \|^2\\
 \geq \left( k_1 \lambda_{min}(\Gamma^i_{av})- \beta J_{max} \right)\| w_1 \|^2 + \left(  \frac{1}{2}  J_{min} - \beta J_{max} \right) \| w_2 \|^2
\end{multline*} Thus for $\beta$ small enough, $V$ is positive definite for all $J$'s satisfying Assumption \ref{uncetain_inertia}. 
Moreover, the following holds
\begin{multline*}
\dot V(w_1,w_2) =-2 \beta k_1 w_1^T \Gamma^i_{av} w_1 - 2 \beta k_2 w_1^T \Gamma^i_{av} w_2 - k_2 w_2^T \Gamma^i_{av} w_2 +\beta k_2 w_2^T J w_2\\
\leq -2 \beta k_1 \lambda_{min}(\Gamma^i_{av}) \| w_1 \|^2 + 2 \beta k_2 \|  \Gamma^i_{av}\| \| w_1 \| \| w_2 \| - k_2  \lambda_{min}(\Gamma^i_{av}) \| w_2 \|^2+\beta k_2  \| J \| \| w_2 \|^2\\
\end{multline*}
Use the following Young's inequality
\begin{equation} \label{young_ineq}
2 \| w_1 \| \| w_2 \| \leq \frac{k_1 \lambda_{min}(\Gamma^i_{av})}{k_2 \|  \Gamma^i_{av}\|} \| w_1 \|^2 + \frac{k_2 \|  \Gamma^i_{av}\|}{k_1 \lambda_{min}(\Gamma^i_{av})} \| w_2 \|^2
\end{equation} so to obtain
\begin{equation} \label{V_dot}
\dot V(w_1,w_2) \leq - \beta \lambda_{min}(\Gamma^i_{av}) \| w_1 \|^2 -
\left[  k_2  \lambda_{min}(\Gamma^i_{av}) - \beta \left( \frac{k_2^2 \|  \Gamma^i_{av}\|}{k_1 \lambda_{min}(\Gamma^i_{av})} + J_{max}\right) \right] \| w_2 \|^2
\end{equation} Thus, for $\beta$ small enough $\dot V$ is negative definite and system (\ref{wdot}) is exponentially stable for all $J$'s satisfying Assumption \ref{uncetain_inertia}. Then, fix $\beta$ so that for all $J$'s that satisfy Assumption \ref{uncetain_inertia}, $V$ is positive definite and $\dot V$ is negative definite, and
 rewrite the Lyapunov function $V$ (see (\ref{def_V})) in the following compact form 
\begin{equation*}
V(w_1,w_2)=w^T P w
\end{equation*}
 where clearly
\begin{equation*}
P=\left[ \begin{array}{cc}
 k_1 \Gamma^i_{av} & \beta J \\ 
\beta J  &  \dfrac{1}{2}J
\end{array} \right]
\end{equation*} Then, note that the following holds
\begin{equation} \label{bound_on_P}
\| P \| \leq k_P
\end{equation}
with
\begin{equation}
k_P=\sqrt{3}\ \left[  k_1 \| \Gamma^i_{av} \| + \left( 2 \beta+\frac{1}{2} \right) J_{max}\right] 
\end{equation}
Moreover, from equation (\ref{V_dot}) it follows immediately that there exists $k_V>0$ such that
\begin{equation} \label{bound_on_V}
\dot V(w_1,w_2) = 2 w^T P A_{av} w \leq - k_V \| w \|^2
\end{equation}
 Now for system (\ref{perturb_syst}) consider the same Lyapunov function $V$ used for system (\ref{wdot}); the derivative of $V$ along the trajectories of  (\ref{perturb_syst})  is given by
\begin{equation*}
\dot V (w_1,w_2)=  \epsilon [2 w^T P A_{av} w+  2 \epsilon w^T P H(t,\epsilon)w]
\end{equation*}
Thus, using (\ref{bound_on_H}) (\ref{bound_on_P}) (\ref{bound_on_V}) we obtain that for $\epsilon$ small enough the following holds
\begin{equation*}
\dot V(w_1,w_2) \leq  \epsilon [-k_V +2 \epsilon k_P k_H] \| w \|^2 
\end{equation*} Thus for $\epsilon$ sufficiently small system (\ref{perturb_syst}) is exponentially stable. As a result, for the same values of $\epsilon$ equilibrium $(q_v,\omega)=(0,0)$ is  exponentially stable for  (\ref{linearized_closed_loop_reduced}), and consequently $(q_v,\omega)=(0,0)$ is locally exponentially stable for the nonlinear system (\ref{closed_loop_reduced}). From equation (\ref{z14}) it follows that given $d <1$, there exists $L>0$ such that
\begin{equation*}
\left| q_{4}-1 \right| \leq L \| q_{v} \|   \ \ \ \forall\ \| q_{v} \| < d
\end{equation*}
Thus, exponential stability of $(q,\omega)=(\bar q,0)$ for (\ref{dynamic_o}) (\ref{control_u}) can be easily obtained.

\end{proof}

\begin{remark} \label{robustness}
Given an inertia matrix $J$ it is relatively simple to show that there exists $\epsilon^*>0$ such that setting $0<\epsilon <\epsilon^*$ the closed-loop system (\ref{dynamic_o}) (\ref{control_u}) is locally exponentially stable at $(q,\omega)=(\bar q, 0)$ \footnote{The actual computation of $\epsilon^*$ is not trivial most of the times (see for example \cite{Rossa:2012ly}).}. It turns out that the value of $\epsilon^*>0$ depends on  $J$; consequently, if $J$ is uncertain, $\epsilon^*$ cannot be determined. However, the previous Theorem has shown that even in the case of unkown $J$, if  bounds $J_{min}$ and $J_{max}$ on its principal moments of inertia are known, then it is possible to determine an $\epsilon^*>0$ such that picking $0<\epsilon <\epsilon^*$ local exponential stability is guaranteed for all $J$'s satisfying those bounds.
\end{remark}

\begin{remark} \label{discuss_assumption1}
Assumption \ref{assumption_gamma_i}  represents an average controllability condition in the following sense. Note that, as a consequence of the fact  that magnetic torques can only be perpendicular to the geomagnetic field, it occurs that 
matrix $\Gamma^i(t)$ is singular for each $t$ since $\Gamma^i(t) B^i(t)=0$ (see (\ref{gamma_i})); thus, system (\ref{dynamic_o}) is not fully controllable at each time instant; as a result, having $\det(\Gamma^i_{av}) \neq 0$ can be interpreted as the ability in the average system to apply magnetic torques in any direction.
\end{remark}


\begin{remark}
The obtained robust stability result hold even if saturation on magnetic moments is taken into account by replacing control (\ref{m_coils}) with
\begin{equation} \label{m_coils_sat}
m_{coils}=m_{coils\ max}\  \text{sat}\left( \frac{1}{m_{coils\ max}} B^b(q,t) \times u \right)
\end{equation} where $m_{coils\ max}$ is the saturation limit on each magnetic  moment, and $\text{sat}: \mathbb{R}^3 \rightarrow \mathbb{R}^3$ is the standard saturation function defined as follows; given $x \in \mathbb{R}^3$, the $i$-th component of $\text{sat}(x)$
is equal to $x_i$ if $|x_i| \leq 1$, otherwise it is equal to either 1 or -1 depending on the sign of $x_i$. The previous theorem still holds because saturation does not modify the linearized system (\ref{linearized_closed_loop_reduced}).
\end{remark}

\begin{remark} \label{choosing_gains}
In practical applications values for gains $k_1$, $k_2$ can be chosen by trial and error following standard guidelines used in proportional-derivative control. For selecting $\epsilon$ in principle we could proceed as follows; determine $\epsilon^*$ by following the procedure presented in the previous proof and pick $0<\epsilon<\epsilon^*$. However, if it is too complicated to follow that approach, an appropriate value for $\epsilon$ could be found by trial and error as well.
\end{remark}

\subsection{Output feedback}

Being able to achieve stability without using attitude rate measures is important from a practical point of view since rate gyros consume power and  increase cost and weight more than the devices needed to implement extra control logic.

In the following thorem we propose a dynamic output (i.e. attitude only) feedback  that is obtained as a simple modification of the output feedback presented in  \cite{Lovera:2004vn}. As in the case of state feedback, the important property that is achieved through such modification is robustness with respect to uncertainties on the inertia matrix.

\begin{theorem} \label{output_feedback}
Consider the magnetically actuated spacecraft 
described by (\ref{dynamic_o}) with uncertain inertia matrix $J$ satisfying Assumption \ref{uncetain_inertia}. Apply the following dynamic attitude feedback control law
\begin{equation} \label{output_feedback_law}
\begin{array}{rcl}
\dot \delta &=& \alpha (q- \epsilon \lambda \delta)\\
u &=& -\epsilon^2 \left( k_1 q_v +  k_2 \alpha \lambda W(q)^{T} (q-\epsilon \lambda \delta) \right)  
\end{array}
\end{equation}
with $\delta \in \mathbb{R}^4$, $k_1>0$, $k_2>0$, $\alpha>0$, and $\lambda>0$. Then, under Assumption \ref{assumption_gamma_i}, there exists $\epsilon^*>0$ such that for any $0<\epsilon<\epsilon^*$,  equilibrium $(q,\omega,\delta)=(\bar q, 0,\frac{1}{\epsilon \lambda} \bar q)$ is locally exponentially stable for (\ref{dynamic_o}) (\ref{output_feedback_law}).
\end{theorem}

\begin{proof}

In order to prove local exponential stability of equilibrium $(q,\omega,\delta)=(\bar q, 0,\frac{1}{\epsilon \lambda} \bar q)$, it suffices considering the restriction of  (\ref{dynamic_o}) (\ref{output_feedback_law}) to the open set $\mathbb{S}^{3+} \times \mathbb{R}^3 \times \mathbb{R}^4$ where $\mathbb{S}^{3+}$ was defined in (\ref{S3}); the latter restriction is given by the following reduced order system
\begin{equation} \label{closed_loop_reduced_o}
\begin{array}{rcl}
  \dot q_{v} &=& W_v(q_{v}) \omega\\
  J \dot \omega &=&  \omega^{\times} J \omega -  \epsilon^2 A_v(q_v) \Gamma^i(t) A_v(q_v)^T \left( k_1 q_v + k_2 \alpha \lambda W_r(q_v)^{T} \left(
  \left[
  \begin{array}{c}
  q_v\\
  (1-q_{v}^T q_{v})^{1/2}
\end{array}
  \right]
  -\epsilon \lambda \delta \right) \right)\\
  \dot \delta &=& \alpha \left( \left[
  \begin{array}{c}
  q_v\\
  (1-q_{v}^T q_{v})^{1/2}
\end{array}
  \right]- \epsilon \lambda \delta \right)
\end{array}
\end{equation}
where $W_v(q_v)$ and $A_v(q_v)$ were defined in equations (\ref{w_v}) and (\ref{a_v}) respectively and $W_r(q_v)$ is defined by to
\begin{equation*}
W_r(q_v)=\dfrac{1}{2}\left[\begin{array}{c}(1-q_{v}^T q_{v})^{1/2} I + q_v^{\times} \\-q_v^T\end{array}\right]
\end{equation*}

Partition $\delta\in \mathbb{R}^4$ as follows 
\begin{equation*}
\delta=[\delta_v^T\  \delta_4]^T
\end{equation*}
where clearly $\delta_v \in \mathbb{R}^3$, and consider the linear approximation of (\ref{closed_loop_reduced_o})  around $(q_v,\omega,\delta_v,\delta_4)=(0, 0,0,\frac{1}{\epsilon \lambda})$ which is given by
\begin{equation} \label{linearized_closed_loop_reduced_o}
\begin{array}{rcl}
  \dot q_{v} &=& \dfrac{1}{2} \omega\\
  J \dot \omega &=& -  \epsilon^2 \Gamma^i(t) \left( k_1 q_v + \dfrac{1}{2} k_2 \alpha \lambda (q_v-\epsilon \lambda \delta_v) \right)\\
  \dot{\delta}_v &=& \alpha (q_v- \epsilon \lambda \delta_v)\\
  \dot{\tilde{\delta}}_4 &=& -\alpha \epsilon \lambda \tilde \delta_4\\
\end{array}
\end{equation}
where $\tilde \delta_4=\delta_4-\frac{1}{\epsilon \lambda}$.
Introduce the following  state-variables' transformation 
\begin{equation*}
  z_1=q_v \ \ \ \  z_2=\omega /\epsilon \ \ \ \  z_3= q_v -  \epsilon \lambda \delta_v \ \ \ \  z_4=  \tilde \delta_4
\end{equation*}  
with $\epsilon>0$ so that system (\ref{linearized_closed_loop_reduced_o}) is transformed into
\begin{equation} \label{linearized_time_varying_o}
\begin{array}{rcl}
  \dot z_1 &=& \dfrac{\epsilon}{2} z_2\\
  J \dot z_2 &=& - \epsilon \Gamma^i(t) \left( k_1 z_1 + \dfrac{1}{2} k_2 \alpha \lambda  z_3) \right)\\
  \dot z_3 &=& \epsilon  \left( \dfrac{1}{2} z_2- \alpha \lambda z_3 \right)\\
  \dot z_4 &=& -\epsilon  \alpha \lambda z_4
 \end{array}
\end{equation}
and consider the so called time-invariant ``average system'' of (\ref{linearized_time_varying_o})
\begin{equation} \label{averaged_closed_loop_o}
\begin{array}{rcl}
 \dot z_1 &=& \dfrac{\epsilon}{2} z_2\\
  J \dot z_2 &=& - \epsilon \Gamma^i_{av} \left( k_1 z_1 + \dfrac{1}{2} k_2 \alpha \lambda  z_3) \right)\\
  \dot z_3 &=& \epsilon  \left( \dfrac{1}{2} z_2- \alpha \lambda z_3 \right)\\
  \dot z_4 &=& -\epsilon  \alpha \lambda z_4\\
 \end{array}
\end{equation} where $\Gamma^i_{av}$ was defined in (\ref{gamma_i_av}). Thus, proceeding in a fashion perfectly parallel to the one followed in the proof of Theorem \ref{exponential_stability} it can be shown that through an appropriate coordinate transformation, system (\ref{linearized_time_varying_o}) can be seen as a perturbation of system (\ref{averaged_closed_loop_o}). Note that the correspondent of system (\ref{linearized_time_varying})  is given by
\begin{equation} \label{w_system}
\begin{array}{rcl}
  \dot w_1 &=& \dfrac{1}{2} w_2\\
  J \dot w_2 &=& -  \Gamma^i_{av} \left( k_1 w_1 + \dfrac{1}{2} k_2 \alpha \lambda  w_3) \right)\\
  \dot w_3 &=&  \dfrac{1}{2} w_2- \alpha \lambda w_3\\
  \dot w_4 &=& -  \alpha \lambda w_4
 \end{array}
\end{equation}
Then, use the following Lyapunov function
\begin{equation*}
V_o(w_{1},w_2)= k_1w_{1}^T \Gamma^i_{av} w_{1}+\frac{1}{2}  w_2^T J  w_2+\frac{1}{2} k_2 \alpha \lambda w_3^T \Gamma^i_{av} w_3+\frac{1}{2} w_4^2 + 2 \beta w_1^T J w_2 - 4 \beta w_2^T J w_3
\end{equation*} with $\beta>0$. It is relatively simple to show that if $\beta$ is small enough, then $V_o$ is positive definite for all $J$'s that satisfy Assumption \ref{uncetain_inertia}. 
Moreover, it is easy to derive that for all such $J$'s  the following holds
\begin{multline*}
\dot V_o(w_{1},w_2) \leq -2 \beta \lambda_{min}(\Gamma^i_{av}) \| w_1 \|^2 -\beta J_{min} \| w_2 \|^2 - (k_2 \alpha^2 \lambda^2  - 2 \beta k_2 \alpha \lambda) \lambda_{min}(\Gamma^i_{av}) \| w_3 \|^2\\
 -  \alpha \lambda w_4^2 + |2 \gamma k_1 - \beta k_2 \alpha \lambda | \lambda_{max}(\Gamma^i_{av}) \| w_1 \| \| w_3 \| + 2 \gamma \lambda \alpha J_{max} \| w_2 \| \| w_3 \|
\end{multline*}
Using  Young's inequalities analogous to (\ref{young_ineq}) for the last two mixed terms,, it is easy to obtain that for $\beta>0$ small enough $\dot V_o$ is negative definite for all $J$'s that satisfy Assumption \ref{uncetain_inertia}. Then, the proof can be completed by using arguments similar to those in the proof of Theorem \ref{exponential_stability}. 
\end{proof}

Considerations  similar to Remarks  \ref{robustness} through \ref{choosing_gains} apply to the proposed output feedback; in particular, in practical applications gains $\alpha$ and $\lambda$ are often chosen by trial and error.

\section{Simulations} \label{simulations}

For simulation consider a satellite whose inertia matrix is equal to
\begin{equation} \label{J}
J=\text{diag}[27\ 17 \ 25] \ \text{kg m}^2
\end{equation}
(see \cite{Lovera:2004vn}). The satellite follows a circular near polar orbit ($incl=87^{\circ}$) with orbit altitude of 450 km; the corresponding 
orbital period is about 5600 s. Without loss of generality the right ascension of the ascending node $\Omega$ is set equal to 0, whereas the initial phases $\alpha_0$ (see (\ref{m_hat})) and $\phi_0$ (see (\ref{x_p_y_p})) have been randomly selected and set equal to $\alpha_0=4.54$ rad and $\phi_0=0.94$ rad.

First, check that for the considered orbit Assumption \ref{assumption_gamma_i} is fulfilled. It was shown in Remark \ref{discuss_assumption2} that the assumption is satisfied if $\det(\Gamma^i_{av}) \neq 0$. The determinant of $1/T \int_0^T \Gamma^i(t) dt$ can be computed numerically, and it turns out that it converges to $9.23\ 10^{-28}$ for $T \rightarrow \infty$. It is of interest to compare the latter value with the value $ 9.49\ 10^{-28}$ obtained by using the analytical expression (\ref{det_Gamma_o_av_simple}) which is valid when $\theta_m$ is approximated to $180^{\circ}$.

Consider an initial state characterized by attitude equal to to the target attitude $q(0)=\bar q$, and by the following high initial angular rate
\begin{equation} \label{omega_0}
\omega(0)=[0.02\ \ 0.02\ \ -0.03]^T\ \text{rad/s}
\end{equation}

\subsection{State feedback} 

The controller's parameters of the state feedback control (\ref{control_u}) have been chosen by trial and error as follows $k_1=2 \ 10^{11}$, $k_2=3\ 10^{11}$, $\epsilon=10^{-3}$. In order to test robustness of the designed state feedback with respect to perturbations of the inertia matrix through a Monte Carlo study, it is useful to generate a random set of perturbed inertia matrices having principal moments of inertia that are in between the smallest ($J_{min}=17\ \text{kg m}^2$
) and the largest ($J_{max}=27\ \text{kg m}^2$
) principal moment of inertia of (\ref{J}). Then, each random perturbed inertia matrix has been generated as follows. First a $3 \times 3$ diagonal matrix $J_{pert\ diag}$ has been determined selecting  each diagonal element on the interval $[J_{min}\  J_{max}]$
by means of the pseudo-random number generator  rand() from Matlab\textsuperscript{TM}. Note that  matrix $J_{pert\ diag}$ satisfies the so called triangular inequalities (see \cite[Problem 6.2]{wie98}) because $2 J_{min}>J_{max}$; thus, it actually represents an inertia matrix. Next, a $3 \times 3$ rotation matrix $R$ has been randomly generated by using the  function for Matlab\textsuperscript{TM} random\_rotation() \cite{random_rotation}; finally the desired randomly generated perturbed inertia matrix has been computed as $J_{pert}=R^T J_{pert\ diag} R$.
Note that Theorem \ref{exponential_stability} guarantees that, if parameter $\epsilon=10^{-3}$ has been chosen small enough, then the desired attitude should be acquired even when the inertia matrix is equal to $J_{pert}$.

Simulations were run for the designed state feedback law using for $J$ the nominal value reported in  (\ref{J}) and each of 200 perturbed values randomly generated; the resulting plots are shown in Fig. \ref{statef_nomJ_nosat}. 
\begin{figure}[h] 
\centering
\subfigure{\includegraphics{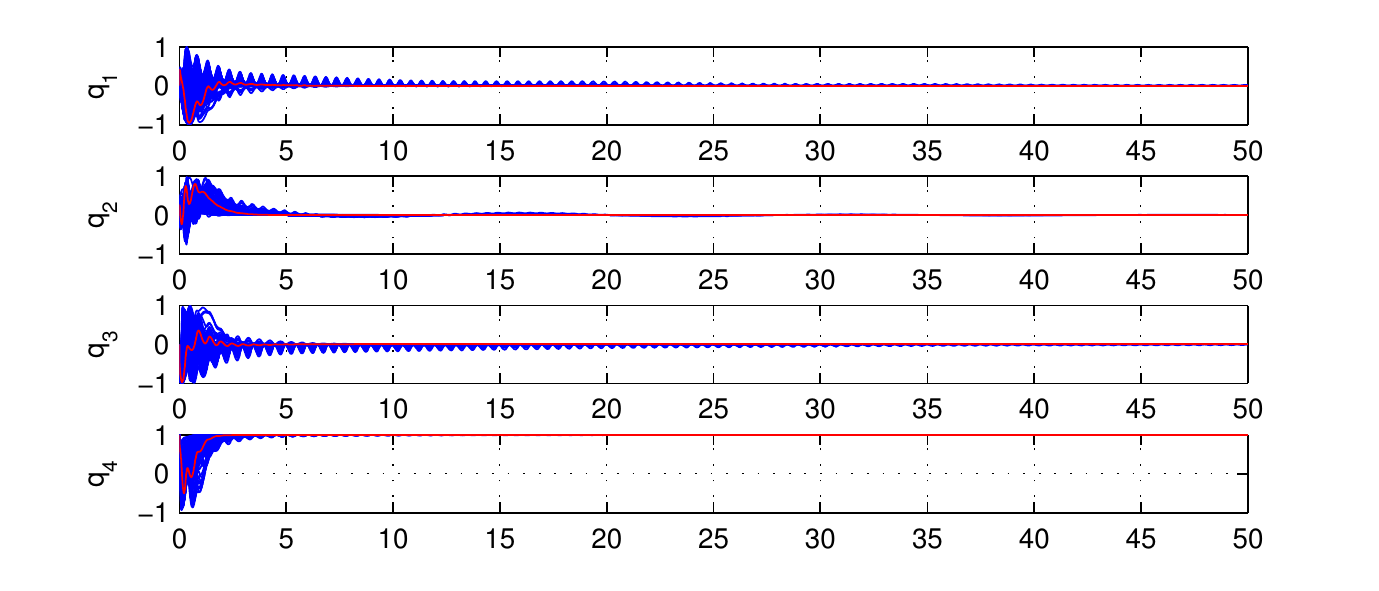}}\\[-6mm]
\subfigure{\includegraphics{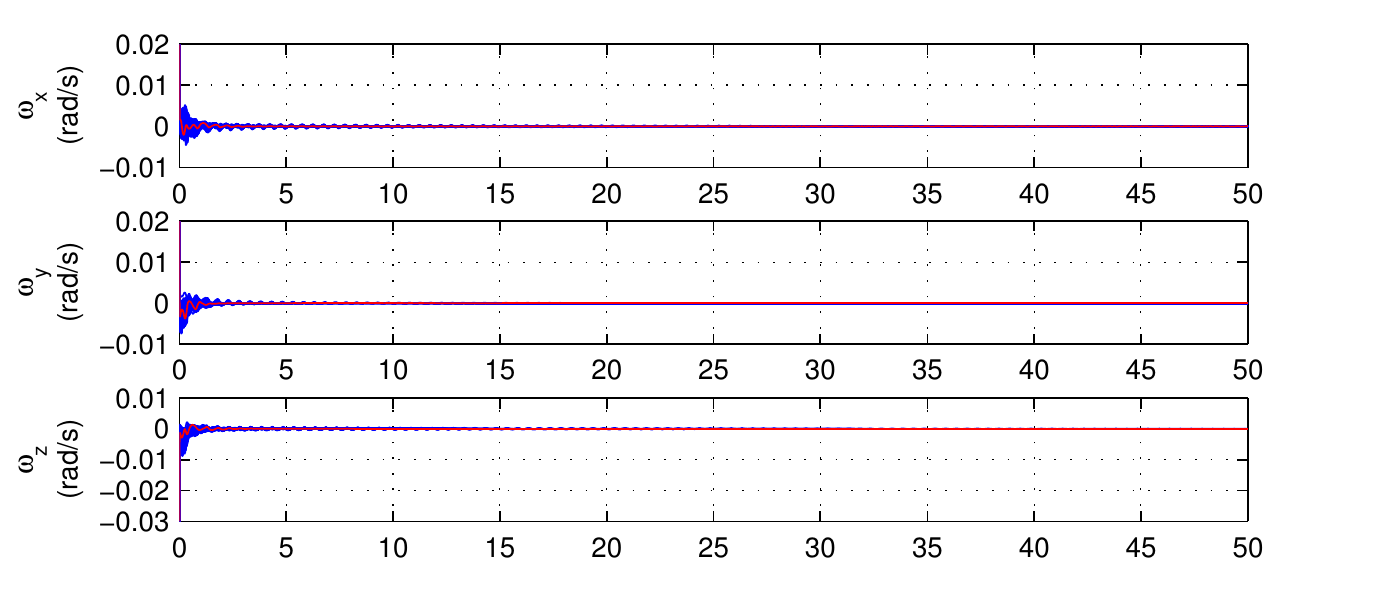}}\\[-5mm]
\subfigure{\includegraphics{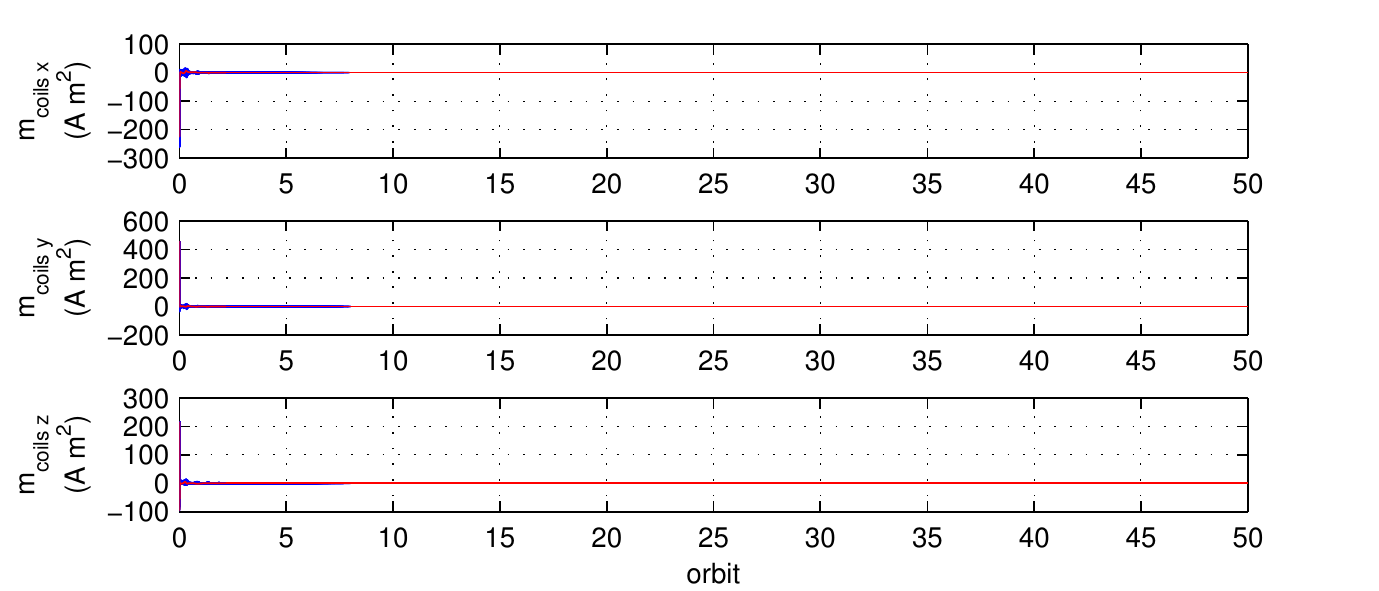}}\\[-6mm]
\caption{Evolutions with state feedback controller. Simulation with nominal inertia matrix (red lines) and Monte Carlo simulations with 200 perturbed inertia matrices (blue envelopes).}
\label{statef_nomJ_nosat}
\end{figure}
Note that asymptotic convergence to the desired attitude is achieved  even with perturbed inertia matrices; however, convergence time can become larger with respect to the nominal case.

\subsection{Output feedback}
 The values of parameters for output feedback (\ref{output_feedback_law}) have been determined by trial and error as follows $k_1=10^{11}$, $k_2=3\ 10^{11}$, $\epsilon=10^{-3}$, $\alpha=4\ 10^3$, $\lambda=1$. Similarly to the state feedback case, simulations were run using the nominal value for $J$ and each of 200 perturbed values which were randomly generated. The results are plotted in Fig. \ref{outputf_nomJ_nosat}.
\begin{figure}[h] 
\centering
\subfigure{\includegraphics{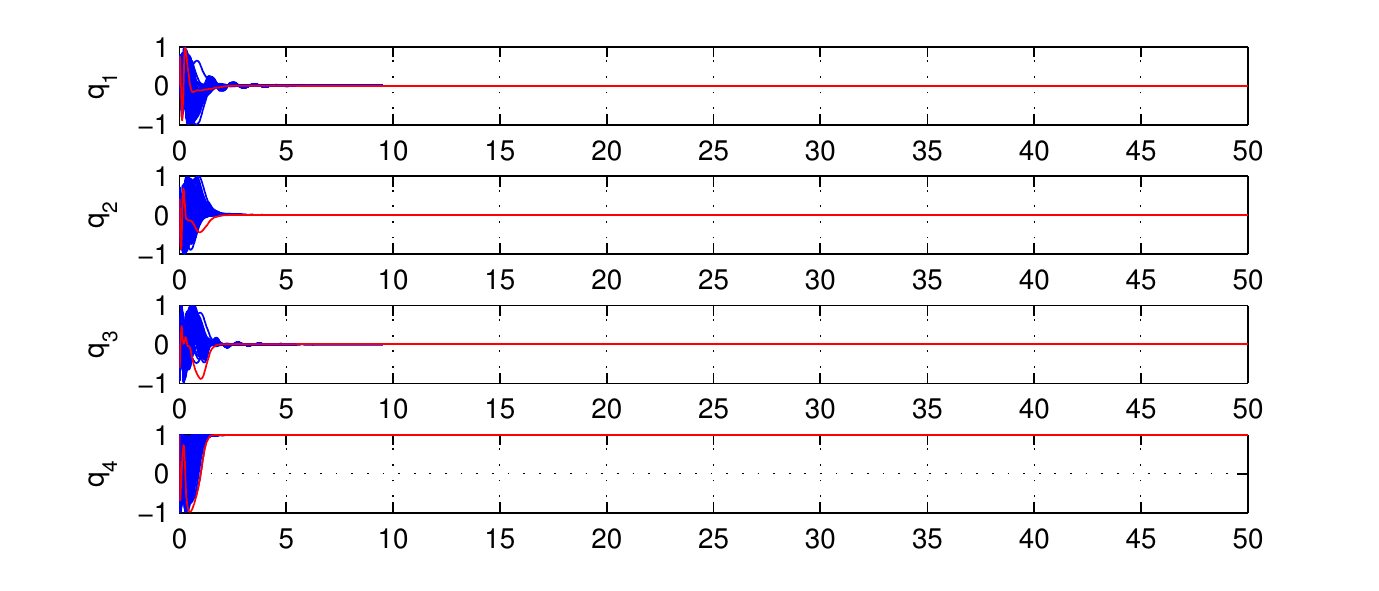}}\\[-6mm]
\subfigure{\includegraphics{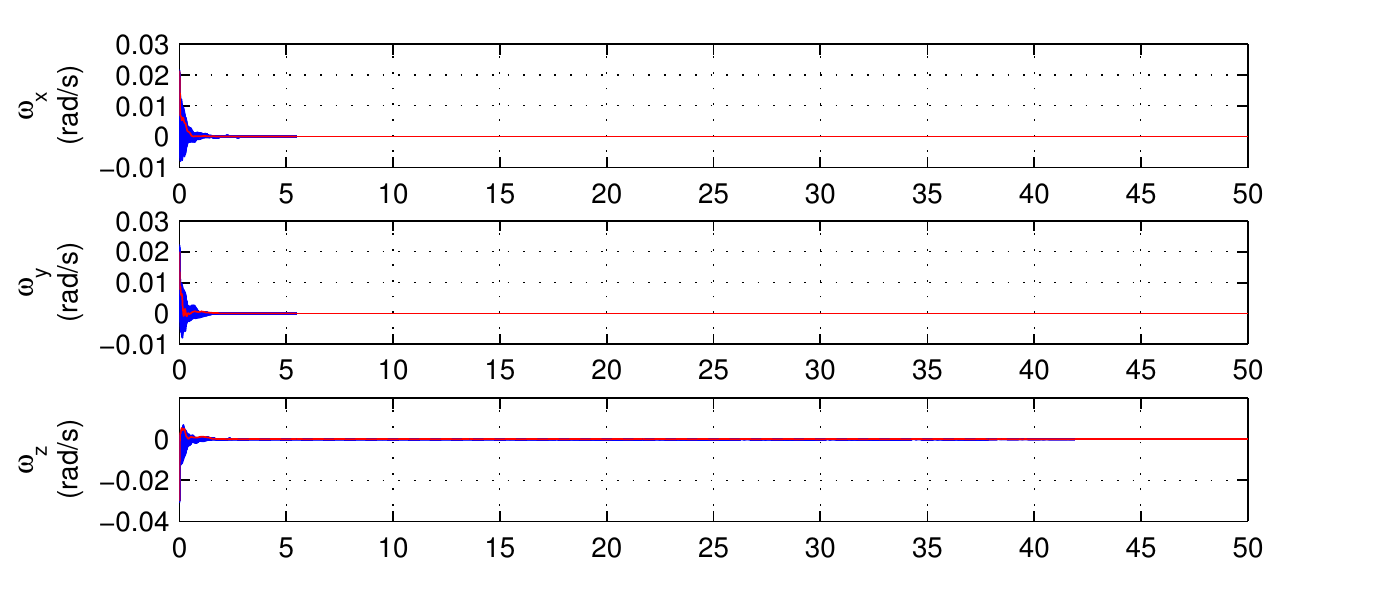}}\\[-5mm]
\subfigure{\includegraphics{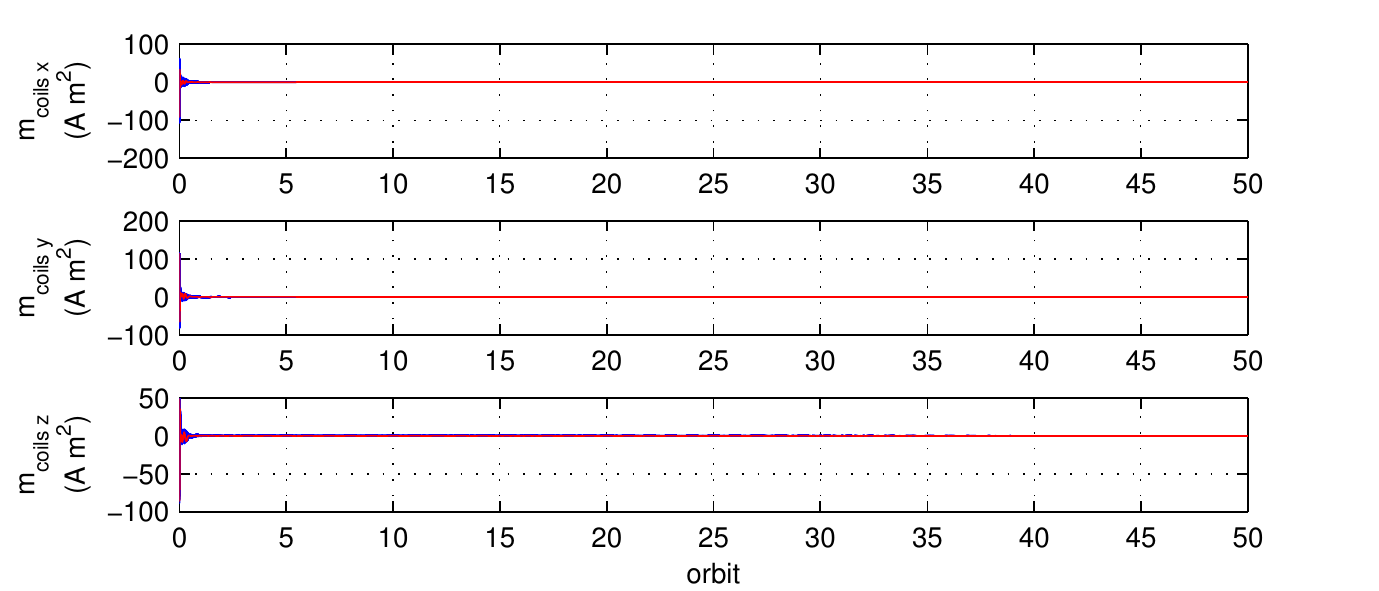}}\\[-6mm]
\caption{Evolutions with output feedback controller. Simulation with nominal inertia matrix (red lines) and Monte Carlo simulations with 200 perturbed inertia matrices (blue envelopes).}
\label{outputf_nomJ_nosat}
\end{figure}
Thus, also in the  output feedback study, it occurs that  asymptotic convergence to the desired attitude is achieved even with perturbed inertia matrices, but convergence time can become larger with respect to the nominal case.

\section{Conclusions}
Three-axis attitude controllers for inertial pointing spacecraft using only magnetorquers have been presented. An attitude plus  attitude rate feedback and an attitude only feedback  are proposed. With both feedbacks local exponential stability and robustness with respect to large inertia uncertainties are achieved. Simulation results have shown the effectiveness of the proposed control designs.

This work shows promising results for further research in the field; in particular, it would be interesting to extend the presented control algorithms to the case of Earth-pointing spacecraft.

\appendix

\section{} \label{appendix}

Recall that given square matrix $X \in \mathbb{R}^{n \times n}$ with eigenvalues inside the unit circle, $I-X$ is invertible and the following holds (see \cite[Lecture 3]{Tyrtyshnikov})
\begin{eqnarray}
(I-X)^{-1} &=& \sum \limits_{i=0}^{\infty} X^i \label{series1}\\
(I-X)^{-2} &=& \sum \limits_{i=1}^{\infty} iX^{i-1} \label{series2}
\end{eqnarray}
From the previous equations the following inequalities are immediatly obtained
\begin{eqnarray}
\| (I-X)^{-1} \| &\leq& \dfrac{1}{1-\| X \|} \label{ineq1}\\[2mm]
\| (I-X)^{-2} \| &\leq& \dfrac{1}{(1-\| X \|^2)^2} \label{ineq2}
\end{eqnarray}

The previous results are useful for proving the following

\begin{lemma} \label{lemma_appendix}
Given $Y \in \mathbb{R}^{n \times n}$ and $\epsilon>0$, if $\epsilon$ is sufficiently small then there exists $0<\tilde{\epsilon}<\epsilon$ such that the  following  holds
\begin{equation*}
(I-\epsilon Y)^{-1}=I+\epsilon Y (I-\tilde{\epsilon} \ Y)^{-2}
\end{equation*} 
\end{lemma}

\begin{proof}
Let $F(\epsilon)=(I-\epsilon Y)^{-1}$. By the mean value theorem, there exists $0<\tilde{\epsilon}<\epsilon$ such that the  following  holds
\begin{equation*}
F(\epsilon)=I+\dfrac{d F}{d \epsilon}(\tilde{\epsilon}) \epsilon
\end{equation*} By using (\ref{series1}) and (\ref{series2}) it follows that for $\epsilon$ small enough
\begin{equation*}
\dfrac{d F}{d \epsilon}(\epsilon)=\dfrac{d}{d \epsilon}\left[ \sum \limits_{i=0}^{\infty} (\epsilon Y)^i \right] = Y \sum \limits_{i=1}^{\infty} i (\epsilon Y)^{i-1}=Y (I-\epsilon \ Y)^{-2}
\end{equation*}
\end{proof}


\bibliographystyle{elsarticle-num}
\bibliography{biblio}

\begin{thebibliography}{10}
\expandafter\ifx\csname url\endcsname\relax
  \def\url#1{\texttt{#1}}\fi
\expandafter\ifx\csname urlprefix\endcsname\relax\def\urlprefix{URL }\fi
\expandafter\ifx\csname href\endcsname\relax
  \def\href#1#2{#2} \def\path#1{#1}\fi

\bibitem{Silani:2005kx}
E.~Silani, M.~Lovera, Magnetic spacecraft attitude control: A survey and some
  new results, Control Engineering Practice 13~(3) (2005) 357--371.

\bibitem{Pittelkau:1993fk}
M.~E. Pittelkau, Optimal periodic control for spacecraft pointing and attitude
  determination, Journal of Guidance, Control, and Dynamics 16~(6) (1993)
  1078--1084.

\bibitem{Lovera:2002uq}
M.~Lovera, E.~De~Marchi, S.~Bittanti, Periodic attitude control techniques for
  small satellites with magnetic actuators, IEEE Transactions on Control
  Systems Technology 10~(1) (2002) 90--95.

\bibitem{pulecchi_cst10}
T.~Pulecchi, M.~Lovera, A.~Varga, Optimal discrete-time design of three-axis
  magnetic attitude control laws, Control Systems Technology, IEEE Transactions
  on 18~(3) (2010) 714--722.

\bibitem{Wisniewski:1999uq}
R.~Wi{\'s}niewski, M.~Blanke, Fully magnetic attitude control for spacecraft
  subject to gravity gradient, Automatica 35~(7) (1999) 1201--1214.

\bibitem{Psiaki:2001yu}
M.~L. Psiaki, Magnetic torquer attitude control via asymptotic periodic linear
  quadratic regulation, Journal of Guidance, Control, and Dynamics 24~(2)
  (2001) 386--394.

\bibitem{Reyhanoglu:2009uq}
M.~Reyhanoglu, C.~Ton, S.~Drakunov, Attitude stabilization of a nadir-pointing
  small satellite using only magnetic actuators, in: Proceedings of the 2nd
  IFAC International Conference on Intelligent Control Systems and Signal
  Processing, Vol.~2, 2009, pp. 292--297.

\bibitem{Lovera:2004vn}
M.~Lovera, A.~Astolfi, Spacecraft attitude control using magnetic actuators,
  Automatica 40~(8) (2004) 1405--1414.

\bibitem{Lovera:2005fk}
M.~Lovera, A.~Astolfi, Global magnetic attitude control of inertially pointing
  spacecraft, Journal of Guidance, Control, and Dynamics 28~(5) (2005)
  1065--1067.

\bibitem{khalil00}
H.~K. Khalil, Nonlinear systems, Prentice Hall, 2002.

\bibitem{sidi97}
M.~J. Sidi, Spacecraft dynamics and control, Cambridge University press, 1997.

\bibitem{wie98}
B.~Wie, Space vehicle dynamics and control, American institute of aeronautics
  and astronautics, 2008.

\bibitem{wertz78}
J.~R. Wertz (Ed.), Spacecraft attitude determination and control, Kluwer
  Academic, 1978.

\bibitem{Rodriguez-Vazquez:2012fk}
A.~L. Rodriguez-Vazquez, M.~A. Martin-Prats, F.~Bernelli-Zazzera, Full magnetic
  satellite attitude control using {ASRE} method, in: 1st IAA Conference on
  Dynamics and Control of Space Systems, 2012.

\bibitem{bullo96}
F.~Bullo, Stability proof for pd controller,
  \url{http://www.cds.caltech.edu/~murray/mlswiki/index.php/First_edition}
  (1996).

\bibitem{Rossa:2012ly}
F.~Della~Rossa, M.~Bergamasco, M.~Lovera, Bifurcation analysis of the attitude
  dynamics for a magnetically controlled spacecraft, in: Proceedings of the
  IEEE Conference on Decision and Control, 2012, pp. 1154--1159.

\bibitem{random_rotation}
M.~Vedenyov, Simplex and random rotation,
  \url{http://www.mathworks.it/matlabcentral/fileexchange/38187-simplex-and-random-rotation/content/random_rotation.m}
  (2012).

\bibitem{Tyrtyshnikov}
E.~E. Tyrtyshnikov, A Brief Introduction to Numerical Analysis, Birkhauser,
  1997.

\end{thebibliography}







\end{document}